\documentclass[12pt,dvipsnames]{article}

\usepackage{amsmath}
\usepackage{amsthm}
\usepackage{amsfonts}
\usepackage{stmaryrd}
\usepackage{amssymb}
\usepackage[all]{xy}
\usepackage{tikz-cd}
\usepackage{comment}
\newcommand\N{\mathbb{N}}
\newcommand\R{\mathbb{R}}

\newcommand{\im}{\textup{im}}

\theoremstyle{plain}
\newtheorem{theoremAlph}{Theorem}
\newtheorem{theorem}{Theorem}[section]
\newtheorem{lemma}[theorem]{Lemma}
\newtheorem{corollary}[theorem]{Corollary}
\newtheorem{proposition}[theorem]{Proposition}

\theoremstyle{definition}
\newtheorem{definition}[theorem]{Definition}

\theoremstyle{remark}
\newtheorem{remark}[theorem]{Remark}

\theoremstyle{definition} 
\newtheorem*{ack}{Acknowledgements}

\usepackage{color}

\usepackage{xcolor}

\usepackage{soul}
\setstcolor{ForestGreen}

\usepackage{hyperref}
\hypersetup{
	colorlinks=true,
	linkcolor=blue,
	citecolor=blue,
	urlcolor=blue
}

\advance \topmargin by -\headheight
\advance \topmargin by -\headsep
     
\evensidemargin \oddsidemargin
\makeatletter
\newcommand{\subjclass}[2][1991]{%
	\let\@oldtitle\@title%
	\gdef\@title{\@oldtitle\footnotetext{#1 \emph{Mathematics subject classification.} #2}}%
}

\begin{document}
	
\title{Intermediate Ricci curvatures and Gromov's Betti number bound}
\date{\today}
\author{Philipp Reiser\thanks{The author acknowledges funding by the Deutsche Forschungsgemeinschaft (DFG, German Research Foundation) -- 281869850 (RTG 2229) and by the Karlsruhe House of Young Scientists (KHYS) -- Research Travel Grant.} \ and David J. Wraith}

%%%%%%%%%%%%%%%%%%%%
	
\maketitle

\begin{abstract}
	We consider intermediate Ricci curvatures $Ric_k$ on a closed Riemannian manifold $M^n$. These interpolate between the Ricci curvature when $k=n-1$ and the sectional curvature when $k=1$. By establishing a surgery result for Riemannian metrics with $Ric_k>0$, we show that Gromov's upper Betti number bound for sectional curvature bounded below fails to hold for $Ric_k>0$ when $\lfloor n/2 \rfloor+2 \le k \le n-1.$ This was previously known only in the case of positive Ricci curvature, \cite{SY1}, \cite{SY2}.
\end{abstract}

\section{Introduction} 
\label{sec:intro}

A basic question in Riemannian geometry is to understand the relationship between curvature and topology. A fundamental result of this type is the following celebrated theorem of Gromov for sectional curvature bounded below:
\begin{theorem}[\cite{Gr}]  
	Given $n \in \N$, $d>0$, $H \in \R$ and a field $\mathbb F$, if $(M^n,g)$ is a closed Riemannian manifold with diameter at most $d$ and sectional curvature $K \ge H,$ then there is a number $C=C(n,Hd^2)$ such that the total Betti number of $M$ with coefficients in $\mathbb F$ is bounded above by $C$, i.e. $$\sum_i b_i(M^n;{\mathbb F}) \le C(n,Hd^2).$$ In particular, if the sectional curvature is nonnegative, the total Betti number is bounded above by a constant depending only on dimension.
\end{theorem}
It is natural to ask to whether such a result continues to hold under weaker curvature conditions. In the case of positive Ricci curvature, Sha and Yang \cite{SY1} were the first to demonstrate that an analogous result is not possible. Indeed one can find examples of closed manifolds in any dimension at least four with positive Ricci curvature and arbitrarily large total Betti number.

Interpolating between the Ricci and sectional curvatures are a natural family of intermediate curvatures: the $k^{th}$-intermediate Ricci curvatures. These are defined as follows:
\begin{definition}\label{def_Ric_k}
Given a point $p$ in a Riemannian manifold $M$, and a collection  $v_0,...,v_k$ of orthonormal vectors in $T_pM$, the $k^{th}$-intermediate Ricci curvature at $p$ corresponding to this choice of vectors is defined to be $\displaystyle{\sum_{i=1}^k K(v_0,v_i)},$ where $K$ denotes the sectional curvature.
\end{definition}
Notice that for an $n$-dimensional manifold $Ric_1>0$ coincides with positive sectional curvature, and $Ric_{n-1}>0$ agrees with positive Ricci curvature. For background concerning these curvatures, see for example \cite{AQZ}, \cite{DGM}, \cite{GX}, \cite{GW1}, \cite{GW2}, \cite{GW3}, \cite{KM}, \cite{Mo1}, \cite{Mo2}, \cite{Sh}, \cite{Wi}, \cite{Wu}.

The primary motivation behind this paper is to understand the extent to which a Gromov-like total Betti number bound holds, or fails to hold, in positive $k^{th}$-intermediate Ricci curvature. We prove:
\begin{theoremAlph}\label{main}
	In dimension $d \ge 5,$ the Gromov Betti number bound fails in $Ric_k>0$ for any $k \ge \lfloor d/2 \rfloor +2.$
\end{theoremAlph}

The principal difficulty in establishing Theorem \ref{main} is the problem of constructing examples with positive $k^{th}$-intermediate Ricci curvature. After many years of progress there is now a rich collection of manifolds which are known to admit positive Ricci curvature, however in contrast, there is a dearth of examples as soon as one increases the strength of the curvature condition. The issue here is the lack of available construction techniques. Before now, the only known approach was via symmetry arguments familiar from the study of positive and non-negative sectional curvature. (See the recent paper \cite{DGM} for examples of this type.) In this paper we develop a surgery-based approach for $Ric_k>0$, where $k$ is roughly greater than half the dimension of the manifold. We apply this technique to establish the following theorem, from which Theorem \ref{main} follows easily:
\begin{theoremAlph}\label{examples}
	For $n,m \ge 2$ with $n \neq m$, and any $r \in \N$, the connected sum $\sharp_{i=1}^r S^n \times S^m$ admits a metric with $Ric_k>0$ for $k \ge \max\{n,m\}+1.$ If $n=m \ge 3$ then this connected sum admits a metric with $Ric_k>0$ for $k \ge n+2.$
\end{theoremAlph}

Note that Theorem \ref{examples} demonstrates that one of the classical results of Sha and Yang for positive Ricci curvature, see \cite{SY2}, continues to hold through a range of stronger curvature conditions.

The precise surgery result used to prove Theorem \ref{examples} is the following, which extends the positive Ricci curvature surgery results by Sha-Yang \cite{SY2} and the authors \cite{Re}, \cite{Wr} to positive $k^{th}$-intermediate Ricci curvature.

For $\rho>0$ we denote by $S^p(\rho)$ the round sphere of radius $\rho$ and for $R,N>0$ we denote by $D^{q+1}_R(N)$ a geodesic ball of radius $R$ in $S^{q+1}(N)$. We will call a map \emph{isometric} if it preserves the Riemannian metrics (and not necessarily the underlying distance functions). 
\begin{theoremAlph}
	\label{surgery}
	Let $(M^n,g_M)$ be a Riemannian manifold with $Ric_k>0$ and, for $p,q\geq2$ with $p+q+1=n$, let $\iota\colon S^p(\rho)\times D_R^{q+1}(N)\hookrightarrow M$ be an isometric embedding. Suppose $k\geq \max\{p,q\}+2$ and $p,q\geq2$. Then there exists a constant $\kappa=\kappa(p,q,k,R/N)>0$ such that if $\frac{\rho}{N}<\kappa$, then the manifold
	\[M_\iota=M\setminus\im(\iota)^\circ\cup_{S^p\times S^q} (D^{p+1}\times S^q) \]
	admits a metric with $Ric_k>0$.
\end{theoremAlph}

The major advantage that the Ricci curvature enjoys that all the stronger curvature conditions under consideration here lack, is that the Ricci curvature (respectively Ricci tensor) is a quadratic form (respectively symmetric bilinear form). These algebraic properties make the positivity of the Ricci curvature relatively straightforward to detect. %\green{Detecting positivity of the sectional curvature is in general more difficult, although the fact that it only depends on the 2-plane spanned by the two input vectors simplifies this problem in many situations. }
On the other hand, detecting $Ric_k>0$ (for $1\le k\le n-2$) is fundamentally more challenging. The detection problem is therefore the first, and most crucial issue which we have to address in this paper. In general this problem is extremely difficult, as is well-known in the case of positive sectional curvature, i.e. $Ric_1>0$. Moreover, the appearance of the vector $v_0$ in every summand in Definition \ref{def_Ric_k} imposes an additional technical difficulty whenever $k>1$.
%especially as $k$ decreases \green{[I don't see how a smaller $k$ makes it more difficult (as long as $k\leq n-2$).]}. 
In the situations we are interested in for our applications, the detection problem reduces to a highly non-trivial algebraic question. In some sense, the resolution of this question (addressed in Proposition \ref{algebra}) is the key innovation in this paper.

The metrics we use in the proof of Theorem \ref{surgery} are doubly warped product metrics as in \cite{SY2}. However, it is not hard to see that the warping functions in \cite{SY2} do not produce a metric with $Ric_k>0$ for $k$ as claimed in the theorem. We therefore follow a modified approach as in \cite{Re}. We also note that the strategy using doubly warped product metrics cannot produce a metric with $Ric_k > 0$ in Theorem \ref{examples} for any smaller $k$ than indicated in the theorem, except in	the case $n=m$, where one could potentially reduce the lower bound by 1, see Remark \ref{simple_optimal} below.

On a more technical level, we also prove a smoothing result via mollification techniques, which we expect to be useful in a variety of situations beyond the current paper.

This paper is laid out as follows. In the second section we discuss the curvature of certain double warped product metrics, and establish the algebraic result which enables us to guarantee that our warped product metrics have $Ric_k>0$ for a range of $k$. In the third section we discuss a gluing result which will allow us to smooth a $C^1$ gluing of double warped product metrics within $Ric_k>0.$ Finally, in the fourth section we discuss the precise choice of scaling functions necessary for us to carry out our surgery construction, and prove the main theorems. 

\begin{ack}
	The first author would like to thank the Department of Mathematics and Statistics of Maynooth University for their hospitality while this work was carried out. Both authors would like to thank the referees for their insightful comments on the original version of this article. 
\end{ack}

%%%%%%%%%%
%%%%%%%%%%

\section{Warped products and positive intermediate Ricci curvatures}
\label{sec:WP}

As discussed in the introduction, we ultimately wish to establish a surgery result for $Ric_k>0.$ The traditional starting point for surgery theorems in positive Ricci curvature is to consider double warped product metrics, and we proceed in a similar fashion here.

Instead of expressing $Ric_k$ as a sum of sectional curvatures, it will be convenient for our purposes to think of $Ric_k$ in terms of the curvature operator ${\mathcal R}\colon\Lambda^2M \to \Lambda^2M.$ Recall that this is defined by the equation $$\langle \mathcal{R}(X \wedge Y),U \wedge W\rangle=R(X,Y,W,U),$$ where $R$ is the curvature tensor and $\langle \cdot,\cdot \rangle$ denotes the Riemannian metric. (Here we assume the following curvature tensor convention: $R(X,Y,W,U)=\langle\nabla_X\nabla_Y W-\nabla_Y\nabla_X W-\nabla_{[X,Y]}W,U\rangle.$) Note that if $V$ is a vector space with inner product $\langle \cdot,\cdot \rangle$, we obtain an induced inner product on $\Lambda^2V$ given by
\[\langle v_1\wedge v_2,v_3\wedge v_4\rangle=\langle v_1,v_3\rangle\langle v_2,v_4\rangle-\langle v_1,v_4\rangle\langle v_2,v_3\rangle. \]
As $R(X,Y,W,U)=R(W,U,X,Y)$ we see that $\mathcal R$ is symmetric. For a pair of orthonormal tangent vectors $U,W$, we have $$K(U,W)=R(U,W,W,U)=\langle \mathcal{R}(U\wedge W),U\wedge W\rangle,$$ so for an orthonormal collection of vectors $v_0,...,v_k$ we have $$\sum_{i=1}^k K(v_0,v_i)=\sum_{i=1}^k \langle \mathcal{R}(v_0 \wedge v_i),v_0\wedge v_i\rangle.$$

Now consider a warped product metric $dt^2+f^2(t)ds^2_p+h^2(t)ds^2_q$ on $\R^+ \times S^p \times S^q$, and orthonormal frame fields $\partial_t, E_1,...,E_p,F_1,...,F_q$ for this metric where $E_1,...,E_p$ are tangent to $S^p$ and $F_1,...,F_q$ are tangent to $S^q$. 

\begin{lemma}[{\cite[Section 4.2.4]{Pet}}]\label{curv_op}
	The curvature operator $\mathcal R\colon \Lambda^2(\R^+ \times S^p \times S^q) \to \Lambda^2(\R^+ \times S^p \times S^q)$ satisfies:
	\begin{align*}
		\mathcal{R}(\partial_t \wedge E_i)&=-\frac{f''}{f}\partial_t\wedge E_i; \\
		\mathcal{R}(\partial_t \wedge F_k)&=-\frac{h''}{h}\partial_t\wedge F_k; \\
		\mathcal{R}(E_i \wedge E_j)&=\frac{1-(f')^2}{f^2}E_i\wedge E_j; \\
		\mathcal{R}(F_k \wedge F_\ell)&=\frac{1-(h')^2}{h^2}F_k\wedge F_\ell; \\
		\mathcal{R}(E_i \wedge F_k)&=-\frac{f'h'}{fh}E_i\wedge F_k. \\
	\end{align*}
	Here it is assumed that $i<j$ and $k<\ell.$
	The elements of $\Lambda^2(\R^+ \times S^p \times S^q)$ listed above form an orthonormal basis of eigenvectors for $\mathcal R$, and the coefficients above form a complete set of eigenvalues.
\end{lemma}

In order to analyse the $Ric_k>0$ condition via the curvature operator in the context of double warped product metrics, we introduce the following concepts.
Fix $k\in\{1,\dots,n-1\}.$

\begin{definition}
	Let $(v_0,\dots,v_k)$ be an orthonormal basis of a $(k+1)$-dimensional subspace of $V$.
	\begin{itemize}
		\item The set $\{v_0\wedge v_1,\dots,v_0\wedge v_k\}\subseteq \Lambda^2 V$ will be called a $k$-\emph{chain}. The vector $v_0$ is the \emph{base} of this $k$-chain.
		\item For a linear map $A\colon \Lambda^2 V\to\Lambda^2 V$ we call
		\[\sum_{i=1}^{k}\langle A(v_0\wedge v_i),v_0\wedge v_i\rangle \]
		the \emph{value} of $A$ on the $k$-chain $\{v_0\wedge v_1,\dots,v_0\wedge v_k\}$.
	\end{itemize}
\end{definition}

We now consider a linear self-adjoint map $A\colon \Lambda^2 V\to \Lambda^2 V$ and suppose that $V$ splits orthogonally as
\[ V=V_1\oplus V_2\oplus V_2 \]
so that the spaces
\[ V_i\wedge V_j \]
with $i,j\in\{1,2,3\}$ are eigenspaces of $A$. We denote the corresponding eigenvalues by $\lambda_{ij}$.
\begin{definition}
	For fixed $i\in\{1,2,3\}$ we say that a partition $n_1+n_2+n_3=k$ of $k$ with $n_j\in\N_0$ is \emph{admissible for $i$}, if $n_i\leq \dim(V_i)-1$ and $n_j\leq \dim(V_j)$ for all $j\in\{1,2,3\}$ with $j\neq i$.
\end{definition}

Our main algebraic result, which is key to estimating $Ric_k$, is the following.

\begin{proposition}\label{algebra}
	The value of $A$ on every $k$-chain is positive if and only if
	\begin{equation}\label{eq:eigenvalues}
		n_1\lambda_{i1}+n_2\lambda_{i2}+n_3\lambda_{i3}>0
	\end{equation}
	for all $i\in\{1,2,3\}$ and all partitions $n_1+n_2+n_3=k$ that are admissible for $i$.
%	Let $A\colon \Lambda^2 V\to \Lambda^2 V$ be a linear self-adjoint map. Suppose that $V$ splits orthogonally as
%	\[V=V_1\oplus V_2\oplus V_3, \]
%	so that the spaces
%	\[V_i\wedge V_j \]
%	with $i,j\in\{1,2,3\}$ are eigenspaces of $A$. Denote the corresponding eigenvalues by $\lambda_{ij}$. If for every $i\in\{1,2,3\}$ and every $n_1,n_2,n_3\in\N_0$ with $n_j\leq \dim(V_j)$ ($j\neq i$), $n_i\leq \dim(V_i)-1$ and $n_1+n_2+n_3=k$ we have
%	\begin{equation}\label{eq:eigenvalues}
%		n_1\lambda_{i1}+n_2\lambda_{i2}+n_3\lambda_{i3}>0,
%	\end{equation}
%	then the value of $A$ on every $k$-chain is positive.
\end{proposition}

Before proving this proposition, we first need some preliminary results. We begin with a definition.
\begin{definition}
	For a linear self-adjoint map $A\colon\Lambda^2 V\to\Lambda^2 V$ and for any unit $v \in V$, let $A_v\colon V \to V$ denote the map given by $$\langle A_v(x),y\rangle=\langle A(v \wedge x),v \wedge y\rangle.$$
\end{definition}
Note that the self-adjointness of $A$ immediately implies that $A_v$ is self-adjoint.

Recall that a self-adjoint linear endomorphism $\theta\colon W \to W$ of a finite dimensional inner product space $(W,\langle\cdot,\cdot\rangle)$ is said to be $k$-positive if for all orthonormal sets of $k$ vectors $\{w_1,...,w_k\}\subset W,$ the sum $$\langle \theta(w_1),w_1\rangle+\cdots +\langle \theta(w_k),w_k\rangle>0.$$
\begin{lemma}[{\cite[Lemma 1.1]{Sha}}]\label{Sha}
	The following are equivalent:
	\begin{enumerate}
		\item $\theta$ is $k$-positive;
		\item the sum of any $k$ eigenvalues of $\theta$ is positive.
	\end{enumerate}
\end{lemma}

Using this result in the case where $\theta=A_v$ we establish:

\begin{lemma}\label{(k+1)-positive}
	The following are equivalent:
	\begin{enumerate}
		\item $A$ is positive on every $k$-chain with base $v$;
		\item $A_v|_{v^\perp}\colon v^\perp\to v^\perp$ is $k$-positive;
		\item $A_v$ is $(k+1)$-positive.
	\end{enumerate}
\end{lemma}

\begin{proof}
	
	First note that $v$ is an eigenvector of $A_v$ with eigenvalue $0$. Since $A_v$ is self-adjoint, the other eigenvalues of $A_v$ (counted with multiplicity) can be associated to an orthonormal frame of eigenvectors belonging to $v^\perp$. Thus, by Lemma \ref{Sha}, $(k+1)$-positivity of $A_v$ is equivalent to $k$-positivity of $A_v|_{v^\perp}$, so items (2) and (3) are equivalent.
	
	The equivalence of items (1) and (2) now directly follows from the fact that every orthonormal set of vectors $\{v_1,\dots,v_k \}\subseteq v^\perp$ defines a $k$-chain $\{v\wedge v_1,\dots,v\wedge v_k \} $ with base $v$ and vice versa.
	
	%	Let us now assume that $A$ is positive on every $k$-chain with base $v$. To show that (2) holds, by Lemma \ref{Sha}, it suffices to show that the sum of any $k$ eigenvalues of $A_v|_{v^\perp}$ is positive. Given any orthonormal $k$-frame $\{e_1,...,e_k\}$ of eigenvectors of $A_v|_{v^\perp},$ we have $$\sum_{i=1}^k \langle A_v|_{v^\perp}(e_i),e_i\rangle=\sum_{i=1}^k \langle A_v(e_i),e_i\rangle=\sum_{i=1}^k \langle A(v \wedge e_i),v\wedge e_i\rangle,$$ and the right-hand side of this expression is positive by hypothesis. Thus $A_v|_{v^\perp}$ is $k$-positive, and hence $A_v$ is $(k+1)$-positive.
	%	
	%	Conversely, assume that $A_v|_{v^\perp}$ is $k$-positive. Let $\{v \wedge w_1,...,v\wedge w_k\}$ be an arbitrary $k$-chain with base $v$. Since $\{w_1,...,w_k\}$ is an orthonormal $k$-frame, by hypothesis we have $$\langle A_v(v),v\rangle+\sum_{i=1}^k \langle A_v(w_i),w_i\rangle >0.$$ But this expression is equal to the value of the $k$-chain, hence $A$ is positive on every $k$-chain with base $v$.
\end{proof}

\begin{proof}[Proof of Proposition \ref{algebra}]
	That condition \eqref{eq:eigenvalues} is necessary can easily be seen by choosing a $k$-chain $\{v_0\wedge v_1,\dots,v_0\wedge v_k \}$ with base $v_0$ contained in $V_i$ and where $n_j$ of the vectors $v_1,\dots,v_k$ are contained in $V_j$.
		
	To see that \eqref{eq:eigenvalues} is sufficient,
	by Lemma \ref{(k+1)-positive}, it suffices to prove that the map $A_v|_{v^\perp}$ is $k$-positive for any unit length $v \in V$. Given such a $v$, there exist unit length vectors $v_i\in V_i$ and $\mu_i\in\R$ with $\mu_1^2+\mu_2^2+\mu_3^2=1$ so that
	\[v=\mu_1 v_1+\mu_2 v_2+\mu_3 v_3. \]
	We extend the vectors $v_i$ to orthonormal bases $(v_i=v_i^1,v_i^2,\dots,v_i^{\dim(V_i)})$ of the spaces $V_i$. Then
	\[(v_1^1,\dots,v_1^{\dim(V_1)},v_2^1,\dots,v_2^{\dim(V_2)},v_3^1,\dots,v_3^{\dim(V_3)}) \]
	is an orthonormal basis of $V$. In this basis the map $A_v$ is given by the following matrix, where we set $a_i=\mu_1^2\lambda_{i1}+\mu_2^2\lambda_{i2}+\mu_3^2\lambda_{i3}$.
	\[
	\begin{tikzpicture}[baseline=(current bounding box.center)]
		\matrix (m) [matrix of math nodes,nodes in empty cells,right delimiter={)},left delimiter={(} ]{
			{\scriptscriptstyle\mu_2^2\lambda_{12}+\mu_3^2\lambda_{13}} & 0 & & 0 & -\mu_1\mu_2\lambda_{12} & 0 & & 0 & -\mu_1\mu_3\lambda_{13}& 0 & & 0\\
			0 & a_1 & & & 0 & & & 0 & 0 & &  &0\\
			& & & 0 & & & & & & & &\\
			0 & & 0 & a_1 & 0 & & & 0 & 0 & & & 0\\
			-\mu_1\mu_2\lambda_{12} & 0 & & 0 & {\scriptscriptstyle \mu_1^2\lambda_{12}+\mu_3^2\lambda_{23}} & 0 & & 0 & -\mu_2\mu_3\lambda_{23} & 0 & & 0\\
			0 & & & 0 & 0 & a_2 & & & 0 & & & 0\\
			& & &  &  & & & 0 & & & & \\
			0 & & & 0 & 0 & & 0 & a_2 & 0 & & & 0\\
			-\mu_1\mu_3\lambda_{13} & 0 & & 0 & -\mu_2\mu_3\lambda_{23} & 0 & & 0 &{\scriptscriptstyle \mu_1^2\lambda_{13}+\mu_2^2\lambda_{23}} & 0 & & 0\\
			0 & & & 0 & 0 & & & 0 & 0 & a_3 & & \\
			& & & & & & & & & & & 0\\
			0 & & & 0 & 0 & & & 0 & 0 & & 0 & a_3\\
		} ;
		\draw[loosely dotted] (m-1-2)-- (m-1-4);
		\draw[loosely dotted] (m-1-2)-- (m-3-4);
		\draw[loosely dotted] (m-1-4)-- (m-3-4);
		
		\draw[loosely dotted] (m-2-1)-- (m-4-1);
		\draw[loosely dotted] (m-2-1)-- (m-4-3);
		\draw[loosely dotted] (m-4-1)-- (m-4-3);
		
		\draw[loosely dotted] (m-2-2)-- (m-4-4);
		
		\draw[loosely dotted] (m-1-6)-- (m-1-8);
		\draw[loosely dotted] (m-2-5)-- (m-2-8);
		\draw[loosely dotted] (m-2-5)-- (m-4-5);
		\draw[loosely dotted] (m-2-5)-- (m-4-8);
		\draw[loosely dotted] (m-2-8)-- (m-4-8);
		\draw[loosely dotted] (m-4-5)-- (m-4-8);
		
		\draw[loosely dotted] (m-1-10)-- (m-1-12);
		\draw[loosely dotted] (m-2-9)-- (m-2-12);
		\draw[loosely dotted] (m-2-9)-- (m-4-9);
		\draw[loosely dotted] (m-2-9)-- (m-4-12);
		\draw[loosely dotted] (m-2-12)-- (m-4-12);
		\draw[loosely dotted] (m-4-9)-- (m-4-12);
		
		\draw[loosely dotted] (m-5-6)-- (m-5-8);
		\draw[loosely dotted] (m-5-6)-- (m-7-8);
		\draw[loosely dotted] (m-5-8)-- (m-7-8);
		
		\draw[loosely dotted] (m-6-5)-- (m-8-5);
		\draw[loosely dotted] (m-6-5)-- (m-8-7);
		\draw[loosely dotted] (m-8-5)-- (m-8-7);
		
		\draw[loosely dotted] (m-6-6)-- (m-8-8);
		
		\draw[loosely dotted] (m-5-2)-- (m-5-4);
		\draw[loosely dotted] (m-6-1)-- (m-6-4);
		\draw[loosely dotted] (m-6-1)-- (m-8-1);
		\draw[loosely dotted] (m-6-1)-- (m-8-4);
		\draw[loosely dotted] (m-6-4)-- (m-8-4);
		\draw[loosely dotted] (m-8-1)-- (m-8-4);
		
		\draw[loosely dotted] (m-5-10)-- (m-5-12);
		\draw[loosely dotted] (m-6-9)-- (m-6-12);
		\draw[loosely dotted] (m-6-9)-- (m-8-9);
		\draw[loosely dotted] (m-6-9)-- (m-8-12);
		\draw[loosely dotted] (m-6-12)-- (m-8-12);
		\draw[loosely dotted] (m-8-9)-- (m-8-12);
		
		\draw[loosely dotted] (m-9-10)-- (m-9-12);
		\draw[loosely dotted] (m-9-10)-- (m-11-12);
		\draw[loosely dotted] (m-9-12)-- (m-11-12);
		
		\draw[loosely dotted] (m-10-9)-- (m-12-9);
		\draw[loosely dotted] (m-10-9)-- (m-12-11);
		\draw[loosely dotted] (m-12-9)-- (m-12-11);
		
		\draw[loosely dotted] (m-10-10)-- (m-12-12);
		
		\draw[loosely dotted] (m-9-6)-- (m-9-8);
		\draw[loosely dotted] (m-10-5)-- (m-10-8);
		\draw[loosely dotted] (m-10-5)-- (m-12-5);
		\draw[loosely dotted] (m-10-5)-- (m-12-8);
		\draw[loosely dotted] (m-10-8)-- (m-12-8);
		\draw[loosely dotted] (m-12-5)-- (m-12-8);
		
		\draw[loosely dotted] (m-9-2)-- (m-9-4);
		\draw[loosely dotted] (m-10-1)-- (m-10-4);
		\draw[loosely dotted] (m-10-1)-- (m-12-1);
		\draw[loosely dotted] (m-10-1)-- (m-12-4);
		\draw[loosely dotted] (m-10-4)-- (m-12-4);
		\draw[loosely dotted] (m-12-1)-- (m-12-4);
	\end{tikzpicture}
	\]
	The eigenvalues of this matrix are the values $a_i$ with multiplicity $\dim(V_i)-1$, together with the eigenvalues of the matrix
	\[
	\begin{pmatrix}
		\mu_2^2\lambda_{12}+\mu_3^2\lambda_{13} & -\mu_1\mu_2\lambda_{12} & -\mu_1\mu_3\lambda_{13} \\
		-\mu_1\mu_2\lambda_{12} & \mu_1^2\lambda_{12}+\mu_3^2\lambda_{23} & -\mu_2\mu_3\lambda_{23} \\
		-\mu_1\mu_3\lambda_{13} & -\mu_2\mu_3\lambda_{23} & \mu_1^2\lambda_{13}+\mu_2^2\lambda_{23}
	\end{pmatrix}.
	\]
	This matrix has eigenvalue $0$ with eigenvector $(\mu_1,\mu_2,\mu_3)^\top$, which corresponds to the vector $v$. The other eigenvalues $\lambda_{\pm}$ are given by
			\[ \lambda_\pm=\frac{1}{2}\left(\mu_1^2(\lambda_{12}+\lambda_{13})+\mu_2^2(\lambda_{12}+\lambda_{23})+\mu_3^2(\lambda_{13}+\lambda_{23})\pm\sqrt{D}\right), \]
		where
		\[ D=-4(\lambda_{12}\lambda_{13}\mu_1^2+\lambda_{12}\lambda_{23}\mu_2^2+\lambda_{13}\lambda_{23}\mu_3^2)+(\mu_1^2(\lambda_{12}+\lambda_{13})+\mu_2^2(\lambda_{12}+\lambda_{23})+\mu_3^2(\lambda_{13}+\lambda_{23}))^2. \]
		Rearranging the terms yields
		\begin{align}
			\notag D=& -4(\mu_1^2+ \mu_2^2+\mu_3^2)(\lambda_{12}\lambda_{13}\mu_1^2+\lambda_{12}\lambda_{23}\mu_2^2+\lambda_{13}\lambda_{23}\mu_3^2)\\
			\notag &+(\mu_1^2(\lambda_{12}+\lambda_{13})+\mu_2^2(\lambda_{12}+\lambda_{23})+\mu_3^2(\lambda_{13}+\lambda_{23}))^2\\
			\notag =&\mu_1^4(\lambda_{12} - \lambda_{13})^2 + \mu_2^4(\lambda_{12} - \lambda_{23})^2 + \mu_3^4(\lambda_{13} - \lambda_{23})^2\\
			\notag & +	2\mu_1^2\mu_2^2(\lambda_{12} - \lambda_{13})(\lambda_{12} - \lambda_{23}) + 2\mu_1^2\mu_3^2(\lambda_{13} - \lambda_{12})(\lambda_{13} - \lambda_{23})\\
			\label{EQ:SQRT_EXPR}&+  2\mu_2^2\mu_3^2(\lambda_{23} - \lambda_{13})(\lambda_{23} - \lambda_{12}).
		\end{align}		
	
%	\begin{align*}
%		&\frac{1}{2}\Big(\mu_1^2(\lambda_{12}+\lambda_{13})+\mu_2^2(\lambda_{12}+\lambda_{23})+\mu_3^2(\lambda_{13}+\lambda_{23})\\
%		&\pm \sqrt{-4(\lambda_{12}\lambda_{13}\mu_1^2+\lambda_{12}\lambda_{23}\mu_2^2+\lambda_{13}\lambda_{23}\mu_3^2)+(\mu_1^2(\lambda_{12}+\lambda_{13})+\mu_2^2(\lambda_{12}+\lambda_{23})+\mu_3^2(\lambda_{13}+\lambda_{23}))^2}  \Big).
%	\end{align*}
%	Rearranging the term under the square root yields
%	\begin{align}
%		\notag-4(&\mu_1^2+ \mu_2^2+\mu_3^2)(\lambda_{12}\lambda_{13}\mu_1^2+\lambda_{12}\lambda_{23}\mu_2^2+\lambda_{13}\lambda_{23}\mu_3^2)\\
%		\notag &+(\mu_1^2(\lambda_{12}+\lambda_{13})+\mu_2^2(\lambda_{12}+\lambda_{23})+\mu_3^2(\lambda_{13}+\lambda_{23}))^2\\
%		\notag =&\mu_1^4(\lambda_{12} - \lambda_{13})^2 + \mu_2^4(\lambda_{12} - \lambda_{23})^2 + \mu_3^4(\lambda_{13} - \lambda_{23})^2\\
%		\notag & +	2\mu_1^2\mu_2^2(\lambda_{12} - \lambda_{13})(\lambda_{12} - \lambda_{23}) + 2\mu_1^2\mu_3^2(\lambda_{13} - \lambda_{12})(\lambda_{13} - \lambda_{23})\\
%		\label{EQ:SQRT_EXPR}&+  2\mu_2^2\mu_3^2(\lambda_{23} - \lambda_{13})(\lambda_{23} - \lambda_{12}).
%	\end{align}
	By symmetry we can assume that $\lambda_{12}\geq\lambda_{13}\geq\lambda_{23}$. Then only the term $2\mu_1^2\mu_3^2(\lambda_{13} - \lambda_{12})(\lambda_{13} - \lambda_{23})$ in \eqref{EQ:SQRT_EXPR} can possibly be negative, all other terms are non-negative. Hence, we can estimate $D$ from below as follows: 
	\begin{align*}
		D&\geq \mu_1^4(\lambda_{12} - \lambda_{13})^2+2\mu_1^2\mu_3^2(\lambda_{13} - \lambda_{12})(\lambda_{13} - \lambda_{23})+\mu_3^4(\lambda_{13} - \lambda_{23})^2\\
		&=(\mu_1^2(\lambda_{12}-\lambda_{13})+\mu_3^2(\lambda_{23}-\lambda_{13}))^2.
	\end{align*}
	Further, by using that
		\[2\mu_1^2\mu_3^2(\lambda_{13} - \lambda_{12})(\lambda_{13} - \lambda_{23})\leq 0\leq 2\mu_1^2\mu_3^2(\lambda_{12} - \lambda_{13})(\lambda_{13} - \lambda_{23}),  \]
	we can estimate $D$ from above as follows:
	\begin{align*}
		D\leq\,\, &\mu_1^4(\lambda_{12} - \lambda_{13})^2 + \mu_2^4(\lambda_{12} - \lambda_{23})^2 + \mu_3^4(\lambda_{13} - \lambda_{23})^2\\
		& +
		2\mu_1^2\mu_2^2(\lambda_{12} - \lambda_{13})(\lambda_{12} - \lambda_{23}) + 2\mu_1^2\mu_3^2(\lambda_{12} - \lambda_{13})(\lambda_{13} - \lambda_{23})\\ &+ 2\mu_2^2\mu_3^2(\lambda_{23} - \lambda_{13})(\lambda_{23} - \lambda_{12})\\
		=\,\,&(\mu_1^2(\lambda_{12}-\lambda_{13})+\mu_2^2(\lambda_{12}-\lambda_{23})+\mu_3^2(\lambda_{13}-\lambda_{23}))^2.
	\end{align*}
	Using these estimates for $D$, we obtain the following lower and upper bounds for $\lambda_+$ and $\lambda_-$, respectively:
	\begin{align*}
		\lambda_+&\geq \frac{1}{2}\left(\mu_1^2(\lambda_{12}+\lambda_{13})+\mu_2^2(\lambda_{12}+\lambda_{23})+\mu_3^2(\lambda_{13}+\lambda_{23})+\mu_1^2(\lambda_{12}-\lambda_{13})+\mu_3^2(\lambda_{23}-\lambda_{13}) \right)\\
		&=\mu_1^2\lambda_{12}+\frac{1}{2}\mu_2^2(\lambda_{12}+\lambda_{23})+\mu_3^2\lambda_{23}\\
		&\geq\mu_1^2\lambda_{12}+\mu_2^2\lambda_{23}+\mu_3^2\lambda_{23}
	\end{align*}
	and
	\begin{align*}
		\lambda_-\geq& \frac{1}{2}(\mu_1^2(\lambda_{12}+\lambda_{13})+\mu_2^2(\lambda_{12}+\lambda_{23})+\mu_3^2(\lambda_{13}+\lambda_{23})\\
		&-(\mu_1^2(\lambda_{12}-\lambda_{13})+\mu_2^2(\lambda_{12}-\lambda_{23})+\mu_3^2(\lambda_{13}-\lambda_{23})) )\\
		=&\mu_1^2\lambda_{13}+\mu_2^2\lambda_{23}+\mu_3^2\lambda_{23}. 
	\end{align*}
	Now let $(\lambda_1,\dots,\lambda_k)$ be a collection of $k$ non-zero eigenvalues of $A_v|_{v^\perp}$. We distinguish several cases:
	\begin{itemize}
		\item \textbf{Case 1:} $\lambda_+$ and $\lambda_-$ are both not contained in this list. Then there are $0\leq n_i\leq\dim(V_i)-1$ with $n_1+n_2+n_3=k$, so that $n_i$ of these eigenvalues are given by $a_i$. Thus,
		\[\sum_{i=1}^k\lambda_i=\sum_{j=1}^{3}\mu_j^2(n_1\lambda_{j1}+n_2\lambda_{j2}+n_3\lambda_{j3})>0. \]
		\item \textbf{Case 2:} Only $\lambda_+$ is contained in this list. Then there are $0\leq n_i\leq\dim(V_i)-1$ with $n_1+n_2+n_3=k-1$, so that $n_i$ of the remaining eigenvalues are given by $a_i$.
		Thus,
		\begin{align*}
			\sum_{i=1}^k\lambda_i\geq&\mu_1^2(n_1\lambda_{11}+(n_2+1)\lambda_{12}+n_3\lambda_{13})+\mu_2^2(n_1\lambda_{21}+n_2\lambda_{22}+(n_3+1)\lambda_{23})\\
			&+\mu_3^2(n_1\lambda_{31}+(n_2+1)\lambda_{32}+n_3\lambda_{33})>0.
		\end{align*}
		\item \textbf{Case 3:} Only $\lambda_-$ is contained in this list. Then there are $0\leq n_i\leq\dim(V_i)-1$ with $n_1+n_2+n_3=k-1$, so that $n_i$ of the remaining eigenvalues are given by $a_i$.
		Thus,
		\begin{align*}
			\sum_{i=1}^k\lambda_i\geq&\mu_1^2(n_1\lambda_{11}+n_2\lambda_{12}+(n_3+1)\lambda_{13})+\mu_2^2(n_1\lambda_{21}+n_2\lambda_{22}+(n_3+1)\lambda_{23})\\
			&+\mu_3^2(n_1\lambda_{31}+(n_2+1)\lambda_{32}+n_3\lambda_{33})>0.
		\end{align*}
		\item \textbf{Case 4:} Both $\lambda_+$ and $\lambda_-$ are contained in this list. Then there are $0\leq n_i\leq\dim(V_i)-1$ with $n_1+n_2+n_3=k-2$, so that $n_i$ of the remaining eigenvalues are given by $a_i$. Further, we have
		\[\lambda_+ +\lambda_-=\mu_1^2(\lambda_{12}+\lambda_{13})+\mu_2^2(\lambda_{12}+\lambda_{23})+\mu_3^2(\lambda_{13}+\lambda_{23}). \]
		Thus,
		\begin{align*}
			\sum_{i=1}^k\lambda_i=&\mu_1^2(n_1\lambda_{11}+(n_2+1)\lambda_{12}+(n_3+1)\lambda_{13})+\mu_2^2((n_1+1)\lambda_{21}+n_2\lambda_{22}+(n_3+1)\lambda_{23})\\
			&+\mu_3^2((n_1+1)\lambda_{31}+(n_2+1)\lambda_{32}+n_3\lambda_{33})>0. 
		\end{align*}
	\end{itemize}
	Hence, the map $A_v|_{v^\perp}$ is $k$-positive. By Lemma \ref{(k+1)-positive} and since $v\in V$ was arbitrary, it follows that $A$ is positive on every $k$-chain.
\end{proof}

%%%%%

	\begin{remark}
		In a similar way one can show in the setting of Proposition \ref{algebra} the value of $A$ on every $k$-chain is greater than $c\in\R$, provided all sums of eigenvalues as in \eqref{eq:eigenvalues} are greater than $c$.
	\end{remark}

\begin{corollary}\label{curvature_criteria}
	Let $I$ be an interval, and consider a warped product metric $dt^2+h^2(t)ds^2_p + f^2(t)ds^2_q$ on $I \times S^p \times S^q$, where $f'' \ge 0,$ $h''<0,$ $h',f' \in [0,1)$, $p \ge 1$ and $q \ge 2.$ Then the warped product metric has $Ric_k>0$ if and only if the following inequalities are satisfied.
	\begin{enumerate}
		\item $-(k-q)\frac{h''}{h}-q\frac{f''}{f}>0$,
		\item $-\frac{h''}{h}+(k-q-1)\frac{1-{h'}^2}{h^2}-q\frac{f'h'}{fh}>0$,
		\item $(k-q)\frac{1-{h'}^2}{h^2}-q\frac{f'h'}{fh}>0$,
		\item $-\frac{f''}{f}-p\frac{f'h'}{fh}+(k-p-1)\frac{1-{f'}^2}{f^2}>0$.
	\end{enumerate}
	%Then $Ric_k>0$.
\end{corollary}
Note that inequality (1) implies $k\geq q+1$ and inequality (4) implies $k\geq p+2$. Thus, the assumptions of Corollary \ref{curvature_criteria} can only be satisfied if $k\geq\max\{p+2,q+1 \}$.

\begin{proof}
	It is easily verified that inequalities (1)--(4) are necessary conditions for having $Ric_k>0$ by considering appropriate $k$-chains consisting of the vectors $\partial_t$, $E_i$ and $F_{\ell}$ in Lemma \ref{curv_op}. In fact, inequality (1) is obtained from a $k$-chain with base $\partial_t$, inequalities (2) and (3) from $k$-chains with base $F_\ell$ and inequality (4) from a $k$-chain with base $E_i$.
	
	To show that the inequalities provide sufficient conditions, we apply our main algebraic result, Proposition \ref{algebra}. For that we consider the vector space $$V=\R\oplus T_xS^p \oplus T_yS^q.$$ According to this result, we need to consider the three cases $i=1,2,3,$ where $i=1$ corresponds to $\R$, $i=2$ to $T_xS^p$, and $i=3$ to $T_yS^q$.
	
	Considering first the case $i=1$, we see that the only possibility for $n_1$ is 0. If $n_1 \neq 0$, this would indicate the existence of a bivector involving a pair of orthonormal vectors from within the $\R$ component of $V$, which of course is impossible. Thus we only have to consider $n_2,n_3 \in \N_0$ with $n_2+n_3=k$. Incorporating the relevant eigenvalues of the curvature operator, we need the following inequality to hold: $$-n_2\frac{h''}{h}-n_3\frac{f''}{f}>0.$$
	Since $h''<0$ and $f''\geq0$, it follows that
	\[-n_2\frac{h''}{h}-n_3\frac{f''}{f}\geq -(k-q)\frac{h''}{h}-q\frac{f''}{f}>0 \]
	by (1).
	
	Moving on to the case $i=2$, we need the following inequality to hold: $$-n_1\frac{h''}{h}+n_2\frac{1-h'^2}{h^2}-n_3\frac{f'h'}{fh}>0.$$ The first two terms in this expression are positive, and the third is negative. Note that since $\text{dim}(\R)=1,$ we have $n_1 \in \{0,1\}.$ If $n_1=1$, then
	\[-n_1\frac{h''}{h}+n_2\frac{1-h'^2}{h^2}-n_3\frac{f'h'}{fh}\geq -\frac{h''}{h}+(k-q-1)\frac{1-h'^2}{h^2}-q\frac{f'h'}{fh}>0 \]
	by (2). If $n_1=0$, then
	\[-n_1\frac{h''}{h}+n_2\frac{1-h'^2}{h^2}-n_3\frac{f'h'}{fh}\geq (k-q)\frac{1-h'^2}{h^2}-q\frac{f'h'}{fh}>0  \]
	by (3).
	
	Finally, we consider $i=3.$ The inequality we need to consider here is $$-n_1\frac{f''}{f}-n_2\frac{f'h'}{fh}+n_3\frac{1-f'^2}{f^2}>0.$$ The only positive term here is the third term, hence we have
	\[-n_1\frac{f''}{f}-n_2\frac{f'h'}{fh}+n_3\frac{1-f'^2}{f^2}\geq -\frac{f''}{f}-p\frac{f'h'}{fh}+(k-p-1)\frac{1-f'^2}{f^2}>0 \]
	by (4).
	
	Taking all three cases together, we will ensure that all three inequalities hold, as required by the hypothesis of our algebraic result.
\end{proof}

%%%%%%%%%%
%%%%%%%%%%

\section{Gluing within positive intermediate Ricci curvature}
\label{sec:glue}

As in \cite{SY2} we wish to perform surgery on $S^{n-1}\times S^{m+1}$ as follows:
$$ S^{n-1} \times \Bigl(S^{m+1}\setminus \coprod_{i=0}^r D^{m+1}_i\Bigr) \cup_{\text{id}} D^n \times \coprod_{i=0}^r S^m_i. \eqno{(\ast)}$$
The resulting manifold is the connected sum $\sharp_r (S^n\times S^m)$, see e.g.\ \cite[Proposition 2.6]{CW} and cf.\ \cite[equation 2]{SY2}. The metric on each connected component $D^n \times S^m$ of $D^n \times \coprod_{i=0}^r S^m_i$ will be given by a double warped product metric of the kind discussed in Section \ref{sec:WP}. The curvature analysis performed in that section will help us guarantee our desired curvature condition. On the left-hand term of $(\ast)$ we will assume the restriction of a certain product metric, namely $\rho^2ds^2_{n-1}+ds^2_{m+1}$ for some suitably small constant $\rho>0.$ It is easy to see that in a neighbourhood of each boundary component, the metric can be described as a double warped product. In the construction of Section \ref{sec:proof} we will arrange for a $C^1$ join between the metrics on each piece of the surgery $(\ast)$. In this section we show that such a metric can always be smoothed within $Ric_k>0$.
%In general we will only be able to arrange for a $C^1$ join between the metrics on each piece of the surgery $(\ast)$, see Section \ref{sec:proof}. Of course we then need to smooth our $C^1$ metric within $Ric_k>0.$ In order to see how to do this, we consider the following situation. 

\begin{lemma}\label{smoothing}
	Consider the following function: $$h(x):=
	\begin{cases}
		f(x) \quad x\le 0 \\
		g(x) \quad x > 0, \\
	\end{cases}
	$$ and assume that $f,g$ are smooth functions on $\R$ such that $h(x)$ is $C^1$ at $x=0$. Then for any $\delta>0$ and any $\nu>0$ sufficiently small (where the upper bound for $\nu$ depends on $f$, $g$ and $\delta$), the function $h$ can be smoothed in the neighbourhood $[-\nu,\nu]$ of $x=0$ in such a way that $h$ and its smoothing are $\delta$-close in a $C^1$ sense, and the second derivative of the smoothed function on $[-\nu,\nu]$ lies in the interval between $\min\{f''(-\nu),g''(\nu)\}-\delta$ and $\max\{f''(-\nu),g''(\nu)\}+\delta.$
\end{lemma}

This result can be proved using a spline interpolation similarly as in the proof of the gluing theorem of Perelman for positive Ricci curvature \cite{Per}, see e.g.\ \cite[Lemma 7]{BWW}. Below we give an alternative proof using mollifying techniques.

\begin{proof}
	Let $\phi_\epsilon\colon\R \to \R$ be a standard mollifying function with support $(-\epsilon,\epsilon)$. Define $\bar{h}$ to be the convolution $\phi_\epsilon \ast h,$ i.e. set $$\bar{h}(x)=\int_\R \phi_\epsilon(x-y)h(y)\,dy.$$ It is well known that $\bar{h}$ is a smooth function. Moreover, since $h$ is $C^1$ we have $\bar{h}'(x)=\phi_\epsilon \ast h'$. Since $h'$ is continuous, we have
	\[ \bar{h}''(x)=\frac{d}{dx}\int_{x-\epsilon}^{x+\epsilon}\phi_\epsilon(x-y)h'(y)dy=\int_{x-\epsilon}^{x+\epsilon}\phi'_\epsilon(x-y)h'(y)dy. \]	
	Since $h'$ is absolutely continuous on $[x-\epsilon,x+\epsilon]$, it follows from integration by parts that the latter expression equals $\phi_\epsilon\ast h''$, where we view $h''$ as a function defined almost everywhere.
	%Since $h''$ is defined and bounded everywhere except at 0, writing $\bar{h}''=\phi_\epsilon \ast h''$ also makes sense.
	
	Given any $\delta>0,$ it is clear that by choosing $\epsilon$ sufficiently small, we can ensure that $\|\bar{h}-h\|_{C^1}<\delta$ over the interval $[-1,1]$ say. Moreover, by choosing $\epsilon$ smaller if necessary, we can ensure that $\|\bar{h}-h\|_{C^2}<\delta$ on the set $[-1,-\epsilon] \cup [\epsilon,1].$ (Notice that this last norm only depends on $f$ alone on $[-1,-\epsilon],$ and on $g$ alone on $[\epsilon,1].$)
	
	As indicated above, let us suppose that the smoothing interval is $[-\nu,\nu]$ for some $\nu\in (0,1)$. Given a choice of $\nu,$ let $\psi\colon \R \to [0,1]$ be a smooth bump function with support in $[-\nu,\nu]$, such that $\psi(x)=1$ for $x \in [-\nu/2,\nu/2]$ say. Set $$H(x)=\bar{h}(x)\psi(x)+h(x)\bigl(1-\psi(x)\bigr).$$ Thus $H$ is a smooth function which agrees with $h$ outside $[-\nu,\nu]$, and with $\bar{h}$ for $x \in [-\nu/2,\nu/2]$. We claim that for a suitable choice of $\epsilon,$ $H$ is the desired smoothing of $h$. 
	
	Clearly $H''$ is influenced by $\psi$ and its first and second derivatives. However for any choice of $\nu$ and $\psi$, it is clear that a sufficiently small value for $\epsilon$ will render both $\|H-h\|_{C^1}<\delta$ on all of $\R$, and $\|H-h\|_{C^2}<\delta$ on $\R \setminus (-\epsilon,\epsilon).$
	
	It remains then to investigate $H''$ over $[-\epsilon,\epsilon].$ Let us assume that $\epsilon<\nu/2$, so then $H$ agrees with $\bar{h}$ on this interval. Thus it suffices to focus on $\bar{h}$. 
	
	We can ensure the desired behaviour of $\bar{h}''$ by choosing $\epsilon$ sufficiently small. Indeed, let $f''_-$ and $f''_+$ (resp.\ $g''_-$ and $g''_+$) be the minimum and maximum value of $f''$ on $[-2\epsilon,0]$ (resp.\ of $g''$ on $[0,2\epsilon]$). Then, since $\bar{h}''=\phi_\epsilon\ast h''$, it follows that for $\epsilon$ sufficiently small,
		\[\bar{h}''\in[\min\{f''_-,g''_-\},\max\{f''_+,g''_+\}] \]
on $[-\epsilon,\epsilon]$. By the continuity of $f''$ and $g''$, this interval is contained in the interval between $\min\{f''(-\nu),g''(\nu)  \}-\delta$ and $\max\{f''(-\nu),g''(\nu)  \}+\delta$ for $\nu$ sufficiently small.
%	To see this, suppose first that over $[-2\epsilon,2\epsilon]\setminus\{0\}$ the second derivative $h''$ is given by the step function $$h''(x)=
%	\begin{cases}
%		f''(0) \quad x < 0 \\
%		g''(0) \quad x > 0, \\
%	\end{cases}
%	$$
%	It is then clear that over $[-\epsilon,\epsilon]$ we will have $\bar{h}''$ interpolating between the values $f''(-\epsilon)$ and $g''(\epsilon).$ For more general $h$, if we choose $\epsilon$ sufficiently small, then $h''$ over $[-2\epsilon,2\epsilon]\setminus\{0\}$ can be approximated arbitrarily closely in a $C^0$ sense by the above expression. In this case, $\bar{h}''$ will not necessarily interpolate between the values $f''(-\epsilon)$ and $g''(\epsilon)$ over $[-\epsilon,\epsilon]$, since $h''$ will in general not be constant on or near $[-\epsilon,\epsilon]\setminus\{0\}$. However, given $\delta>0$ as above, by choosing $\epsilon$ sufficiently small, we can certainly ensure that $\bar{h}''$ remains in the interval $[\min\{f''(-\epsilon),g''(\epsilon)\}-\delta,\max\{f''(-\epsilon),g''(\epsilon)\}+\delta].$
	
	Thus with a suitable choice of $\epsilon$ (depending on the functions $f$ and $g$), the resulting function $H$ will have all the desired properties.
\end{proof}

Adapting the argument given above by replacing $C^1$ by $C^k$ immediately yields the following corollary:
\begin{corollary}
	Suppose that $f,g$ in the above lemma are $C^{l}$ functions for some $l\in\N\cup\{\infty\}$ and that $h$ at $x=0$ is $C^k$ for $k<l$. Then given any $\delta>0$ and any $\nu>0$ sufficiently small, there exists a $C^{l}$ function $H$ which agrees with $h$ outside $[-\nu,\nu]$, satisfies $\|H-h\|_{C^k}<\delta$, and we have that $H^{(k+1)}(x) \in [\min\{f^{(k+1)}(-\nu),g^{(k+1)}(\nu)\}-\delta,\max\{f^{(k+1)}(-\nu),g^{(k+1)}(\nu)\}+\delta]$ for $x \in [-\nu,\nu].$
\end{corollary}

When applied to double warped product metrics, Lemma \ref{smoothing} yields the following corollary:

\begin{corollary}\label{gluing} 
	Consider smooth double warped product metrics $dt^2+h_1^2(t)ds^2_p +f_1^2(t)ds^2_q$ and $dt^2+h_2^2(t)ds^2_p +f_2^2(t)ds^2_q$ on $\R \times S^p \times S^q$. Suppose that for some $k$ the first of these metrics satisfies $Ric_k>0$ when $t\le 0$, and the second satisfies the same curvature condition when $t \ge 0.$ If the metric $$g:=
	\begin{cases}
		dt^2+h_1^2(t)ds^2_p +f_1^2(t)ds^2_q \quad \text{ if } t \le 0 \\
		dt^2+h_2^2(t)ds^2_p +f_2^2(t)ds^2_q \quad \text{ if } t \ge 0 \\
	\end{cases}
	$$
	is $C^1$ at $t=0$, then there exists a smooth metric $dt^2+h^2(t)ds^2_p +f^2(t)ds^2_q$ with $Ric_k>0$ on $\R \times S^p \times S^q$, which agrees with $g$ outside an arbitrarily small neighbourhood of $t=0$.
\end{corollary}

\begin{proof}
	
	We simply apply Lemma \ref{smoothing} to the corresponding pairs of scaling functions, i.e.\ to $f_1$ and $f_2$, and to $h_1$ and $h_2$. Notice that the curvature inequalities in Corollary \ref{curvature_criteria} depend linearly on the second derivatives of the scaling functions. It is then clear that by choosing $\delta$ in Lemma \ref{smoothing} sufficiently small, the resulting smooth functions are the desired scaling functions $f$ and $h$.
	
\end{proof}

%%%%%%%%%%
%%%%%%%%%%

\section{Proof of the main result}
\label{sec:proof}

The main task in this section is to show how to choose scaling functions $f,h$ for our double warped product metric which satisfy the conditions required by Corollary \ref{curvature_criteria}. These functions must also satisfy certain boundary conditions at $t=0$ to ensure smooth extension to a metric on $D^n \times S^m$, as well as boundary conditions for large $t$ so as to guarantee at least a $C^1$ join with the metric on the ambient manifold when surgery is completed. Corollary \ref{gluing} will then complete the metric surgery construction within $Ric_k>0$ for suitable $k$.

\bigskip\bigskip

We begin by restating Theorem \ref{surgery} from the Introduction.
\begin{theorem}
	Let $(M^n,g_M)$ be a Riemannian manifold with $Ric_k>0$ and, for $p,q\geq2$ with $p+q+1=n$, let $\iota\colon S^p(\rho)\times D_R^{q+1}(N)\hookrightarrow M$ be an isometric embedding. Suppose $k\geq \max\{p,q\}+2$. Then there exists a constant $\kappa=\kappa(p,q,k,R/N)>0$ such that if $\frac{\rho}{N}<\kappa$, then the manifold
	\[M_\iota=M\setminus\im(\iota)^\circ\cup_{S^p\times S^q} (D^{p+1}\times S^q) \]
	admits a metric with $Ric_k>0$.
\end{theorem}

\begin{proof}
	Let $I=[t_0,t_1]$ be a closed interval. Then we can identify $D^{p+1}\times S^q$ with the space obtained from
	\[I\times S^p\times S^q \]
	by collapsing $\{t_0\}\times S^p\times \{x\}$ for every $x\in S^q$. Via this identification a double warped product metric
	\[g_{f,h}=dt^2+h^2(t)ds_p^2+f^2(t)ds_q^2 \]
	defines a smooth metric on $D^{p+1}\times S^q$ if and only if
	\begin{align}
		\label{EQ:BOUND_COND1}
		h^{(even)}(t_0)&=0, &h'(t_0)&=1,\\
		\label{EQ:BOUND_COND2}
		f(t_0)&>0, &f^{(odd)}(t_0)&=0,
	\end{align}
	see \cite[Proposition 1.4.7]{Pet}.
	
	The metric on $S^p(\rho)\times D_R^{q+1}(N)$ can also be obtained from the double warped product metric
	\[dt^2+\rho^2 ds_p^2+N^2\sin^2\left(\frac{t}{N}\right)ds_q^2 \]
	on $[0,R]\times S^p\times S^q$. We replace $R$ by $\frac{9}{10}R$, so that a neighbourhood of the boundary of $M\setminus\iota(S^p(\rho)\times D_{R}^{q+1}(N))$ is isometric to the corresponding part of this double warped product metric. Hence, in order to make a $C^1$ join with this metric, the functions $f$ and $h$ need to satisfy
	\begin{align}
		\label{EQ:BOUND_COND3}
		h(t_1)&=\rho,&h'(t_1)&=0,\\
		\label{EQ:BOUND_COND4}
		f(t_1)&=N\sin\left(\frac{R}{N}\right),&f'(t_1)&=\cos\left(\frac{R}{N}\right).
	\end{align}
	Thus, if we can construct functions $f$ and $h$ satisfying the boundary conditions \eqref{EQ:BOUND_COND1}--\eqref{EQ:BOUND_COND4} for which the metric $g_{f,h}$ has positive $k$-th intermediate Ricci curvature, then we can apply Lemma \ref{gluing} to glue the metric $g_{f,h}$ to $M\setminus\iota(S^p(\rho)\times D_{R}^{q+1}(N))$ and obtain a metric with positive $k^{th}$-intermediate Ricci curvature on $M_\iota$.
	
	We will construct the required functions piecewise in four steps based on a partition $t_0<0<t_2<t_3<t_1$. We will then use Corollary \ref{curvature_criteria} to show that these functions define a metric of $Ric_k>0$.
	
	We start as in \cite[Section 3.3]{Re}, i.e.\ we first define the functions $h_0\colon [0,\infty)\to \R$ and $f_C\colon [0,\infty)\to\R$ for some $C>0$ as the unique smooth functions satisfying
	\begin{align*}
		h_0'&=e^{-\frac{1}{2}h_0^2},\\
		h_0(0)&=1,
	\end{align*}
	and
	\begin{align*}
		f_C''&=C e^{-h_0^2}f_C,\\
		f_C(0)&=1,\\
		f_C'(0)&=0.
	\end{align*}
	For $a,b>0$ we set $h_a=a\cdot h_0$ and $f_{b,C}=b\cdot f_C$. Then from the definition it follows that
	\begin{align*}
		h_a(0)&=a, &h_a^\prime (0)&=\frac{a}{\sqrt{e}},\\
		f_{b,C}(0)&=b,&f_{b,C}'(0)&=0.
	\end{align*}
	To show that $g_{f_{b,C},h_a}$ has $Ric_k>0$ for suitable $a,b,C$ we will use Corollary \ref{curvature_criteria}. For that, we need to ensure that the inequalities (1)--(4) in Corollary \ref{curvature_criteria} are satisfied. Here we follow the same strategy as in \cite[Proof of Lemma 3.9]{Re}.
	
	First note that $h_0''=-h_0e^{-h_0^2}$. Hence, we have
	\[\frac{h_a''}{h_a}=-e^{-h_0^2}\quad \text{and}\quad \frac{f_{b,C}''}{f_{b,C}}=Ce^{-h_0^2}. \]
	It follows that inequality (1) is satisfied for all $C>0$ sufficiently small and for all $a,b>0$. 
	
	For the other inequalities we have by \cite[Lemma 3.7 (4)]{Re} that
	\[\frac{f_{b,C}' h_a'}{f_{b,C}h_a}=\frac{f_C'}{f_C h_0 h_0'}{h_0'}^2\leq {h_0'}^2. \]
	Hence, it follows that
	\begin{align*}
		-\frac{h_a''}{h_a}+(k-q-1)\frac{1-{h_a'}^2}{h_a^2}-q\frac{f_{b,C}' h_a'}{f_{b,C}h_a}&\geq\frac{1}{h_0^2}\left(h_0^2 e^{-h_0^2}+(k-q-1)\left(\frac{1}{a^2}-e^{-h_0^2}\right)-qh_0^2e^{-h_0^2}  \right)
	\end{align*}
	Since $h_0^2 e^{-h_0^2}$ and $e^{-h_0^2}$ are bounded independently of $a,b,C$, it follows that this expression is positive for all $a>0$ sufficiently small and for all $b,C>0$. Thus, inequality (2) holds. By a similar argument it also follows that inequality (3) holds for all $a>0$ sufficiently small and for all $b,C>0$.
	
	For inequality (4) we fix $a,C>0$ so that inequalities (1)--(3) are satisfied. We have
	\begin{align*}
		-\frac{f_{b,C}''}{f_{b,C}}-p\frac{f_{b,C}'h_a'}{f_{b,C}h_a}+(k-p-1)\frac{1-{f_{b,C}'}^2}{f_{b,C}^2}&\geq\frac{1}{f_C^2}\left(-Ce^{-h_0^2}f_C^2-pf_C^2 {h_0'}^2+(k-p-1)\left(\frac{1}{b^2}-{f_C'}^2  \right)  \right)\\
		&=\frac{1}{f_C^2}\left(-(p+C) f_C^2 {h_0'}^2 +(k-p-1)\left(\frac{1}{b^2}-{f_C'}^2  \right)  \right).
	\end{align*}
	The last expression is positive at $t=0$, and hence also for all $t>0$ sufficiently small. Further, by \cite[Lemma 3.7 (3)]{Re}, the term $f_Ch_0'$ converges to $0$ as $t\to\infty$ and by \cite[Lemma 3.7 (2)]{Re}, the term $f_C'$ converges to $\infty $ as $t\to\infty$. Thus, the whole expression eventually becomes negative as $t\to\infty$. Let $t_b>0$ be the smallest value for which it vanishes. Since $f_C h_0'$ is bounded, we have $t_b\to\infty$ as $b\to0$. Further, rearranging the terms yields
	\[f'_{b,C}(t_b)^2=b^2f_C'(t_b)^2=1-\frac{b^2}{k-p-1}(p+C)f_C(t_b)^2 h_0'(t_b)^2\to1 \]
	as $b\to 0$.
	
	Hence, there is $t_2>0$ so that for $a,b,C$ sufficiently small the inequalities in Corollary \ref{curvature_criteria} are satisfied on $[0,t_2]$, so the metric $g_{f_{b,C},h_a}$ has positive $k^{th}$-intermediate Ricci curvature, with $f_{b,C}'(t_2)>\cos(R/N)$ and $h'_a(t_2)>0$. We set $h=h_a$ and $f=f_{b,C}$ on $[0,t_2]$.
	
	To satisfy the boundary conditions \eqref{EQ:BOUND_COND1} and \eqref{EQ:BOUND_COND2} let $R',N'>0$ so that
	\[N'\sin\left(\frac{R'}{N'}\right)=a,\quad \cos\left(\frac{R'}{N'}\right)=\frac{a}{\sqrt{e}}. \]
	For $t_0=-R'$ we then define for $t\in[t_0,0]$
	\begin{align*}
		h(t)&=N'\sin\left(\frac{t-t_0}{N'}\right),\\
		f(t)&=b.
	\end{align*}
	Then $h$ and $f$ are $C^1$ and satisfy \eqref{EQ:BOUND_COND1} and \eqref{EQ:BOUND_COND2}. Further it is easily verified that the assumptions of Corollary \ref{curvature_criteria} are satisfied on $[t_0,0]$, hence the metric $g_{f,h}$ has positive $k^{th}$-intermediate Ricci curvature on this piece.
	
	For the boundary conditions \eqref{EQ:BOUND_COND3} and \eqref{EQ:BOUND_COND4} we extend $h$ and $f$ in two steps. First, for $\varepsilon>0$ there exists $\delta>0$, so that $h$ can be extended on $[t_2,t_2+\delta]$ such that
	\begin{enumerate}
		\item $h$ is $C^1$ at $t=t_2$ and smooth at all other points,
		\item $h''_+(t_2),h''(t)<-\frac{1}{\varepsilon}$ for all $t\in (t_2,t_2+\delta]$, and
		\item $h'(t_2+\delta)=0$.
	\end{enumerate}
	By integrating the inequality $h''(t)<-1/\varepsilon$ over $[t_2,t_2+\delta]$ and using $h'(t_2+\delta)=0$, we obtain $\delta<\varepsilon h'(t_2)$. Hence,
	\[h(t_2+\delta)<h(t_2)+\delta h'(t_2)<h(t_2)+\varepsilon h'(t_2)^2. \]
	In particular, $\delta=O(\varepsilon)$ and $h(t)= h(t_2)+O(\varepsilon)$. Further, by construction, we have $h(t)\geq h(t_2)$ and $h'(t)\leq h'(t_2)$, and for $\varepsilon$ sufficiently small $h''(t)<h_-''(t_2)$ for $t\in [t_2,t_2+\delta]$. If we now extend $f$ linearly on $[t_2,t_2+\delta]$, then we obtain
	\[f(t)=f(t_2)+O(\varepsilon), \quad f'(t)=f'(t_2)\quad \text{ and }\quad f''(t)=0\,\leq f_-''(t_2)  \]
	for all $t\in[t_2,t_2+\delta]$. Hence, for $\varepsilon$ sufficiently small, it follows that all inequalities in Corollary \ref{curvature_criteria} are satisfied: For (1) this follows from the fact that $f''\equiv0$ and $h''<0$ on $[t_2,t_2+\delta]$. For (2) we estimate
	\begin{align*}
		-\frac{h''(t)}{h(t)}+(k-q-1)\frac{1-h'(t)^2}{h(t)^2}&-q\frac{f'(t)h'(t)}{f(t)h(t)}\\
		&\geq -\frac{h''_-(t_2)}{h(t)}+(k-q-1)\frac{1-h'(t_2)^2}{h(t)^2}-q\frac{f'(t_2)h'(t_2)}{f(t)h(t)}.
	\end{align*}
	Now, by using the fact that $h(t)=h(t_2)+O(\varepsilon)$ and $f(t)=f(t_2)+O(\varepsilon)$ and that the corresponding expression at $t=t_2$ is positive, the required inequality follows for $\varepsilon$ sufficiently small. The remaining inequalities follow similarly. Hence $g_{f,h}$ has positive $k^{th}$-intermediate Ricci curvature on $[t_2,t_2+\delta]$. We set $t_3=t_2+\delta$.
	
	Finally, we extend $f$ on $[t_3,t_1]$ for some $t_1>t_3$, so that
	\begin{enumerate}
		\item $f$ is $C^1$ at $t=t_3$ and smooth at all other points,
		\item $f''_+(t_3),f''(t)<0$ for all $t\in (t_3,t_3+\delta]$,
		\item $\frac{f(t_1)}{h(t_3)}=\frac{N}{\rho}\sin\left(\frac{R}{N}\right)$ and $f'(t_1)=\cos\left(\frac{R}{N}\right)$.
	\end{enumerate}
	This is possible if and only if $f(t_3)<\frac{h(t_3)}{\rho}N\sin\left(\frac{R}{N}\right)$, which can be arranged, by choosing $\varepsilon$ sufficiently small, if and only if
	\[f(t_2)<\frac{h(t_2)}{\rho}N\sin\left(\frac{R}{N}\right), \]
	i.e.\ if and only if
	\[\frac{\rho}{N}<\frac{h(t_2)}{f(t_2)}\sin\left(\frac{R}{N}\right). \]
	By construction, the values of $h$ and $f$ at $t=t_2$ only depend on the quotient $R/N$ and on the coefficients of the differential inequalities in Corollary \ref{curvature_criteria}, which in turn only depend on $p$, $q$ and $k$. Hence we can define
	\[\kappa=\kappa(p,q,k,R/N)=\frac{h(t_2)}{f(t_2)}\sin\left(\frac{R}{N}\right) \]
	and we can extend $f$ in the desired way if and only if
	\[\frac{\rho}{N}<\kappa. \]
	If we now extend $h$ constantly on $[t_3,t_1]$, it follows from Corollary \ref{curvature_criteria} that $g_{f,h}$ has positive $k^{th}$-intermediate Ricci curvature.
	
	We rescale the metric by $\lambda=\frac{\rho}{h(t_1)}$, i.e.\ we replace the functions $h$ and $f$ by
	\[\lambda h\left(\frac{t}{\lambda}\right)\text{ and }\lambda f\left(\frac{t}{\lambda}\right) \]
	and we replace $t_0$ and $t_1$ by $\lambda t_0$ and $\lambda t_1$, respectively. Note that this preserves the boundary conditions \eqref{EQ:BOUND_COND1} and \eqref{EQ:BOUND_COND2} and the boundary conditions \eqref{EQ:BOUND_COND3} and \eqref{EQ:BOUND_COND4} are now also satisfied. Finally, we smooth the functions $h$ and $f$ using Corollary \ref{gluing}. This concludes the proof.
\end{proof}

\begin{remark}
	The metric in Theorem \ref{surgery} is constructed so that it coincides with $D^{p+1}_{R'}(N')\times S^q(\rho')$ near the centre of $D^{p+1}\times S^q$ for some $\rho',R',N'>0$. In fact, by defining a suitable value for $h_0(0)$ and by possibly choosing the constants $a$ and $b$ even smaller, we can prescribe the values of $\rho',R',N'$ (after possibly rescaling the metric). This adds additional dependencies for $\kappa$.
\end{remark}

\medskip

\begin{proof}[Proof of Theorem \ref{examples}]
	Let us compare this with the Sha--Yang surgery result. Sha--Yang guarantee the existence of a positive Ricci curvature metric on $\sharp^r S^n \times S^m$ for any $n,m \ge 2$ by performing $r+1$ surgeries on the $S^{n-1}$ factor of the product $S^{n-1} \times S^{m+1}$ as in ($\ast$). Using Theorem \ref{surgery} instead of Sha and Yang's surgery result, and comparing the two scenarios, we clearly need $p=n-1$ and and $q=m$. Thus we can put a $Ric_k>0$ metric on $\sharp^r S^n \times S^m$ whenever $$k \ge \max\{q+2,p+2\}=\max\{m+2,n+1\}.$$ But notice that $m$ and $n$ are interchangable here, so the smallest value of $k$ which will work is the smaller of the quantities $\max\{m+2,n+1\}$ and $\max\{n+2,m+1\}.$ It is easy to see that if $m \neq n,$ then $k_{min}=\max\{n,m\}+1,$ and if $n=m$ then $k_{min}=n+2.$
\end{proof}

\begin{proof}[Proof of Theorem \ref{main}]
	We want to show that the Gromov Betti number bound fails to hold for some $Ric_k>0$, with $k$ as small as possible, in a given dimension $d \ge 5.$ There are two situations to consider, depending on whether $d$ is even or odd.
	
	Beginning with $d$ even, suppose that $d=2n.$ Considering $\sharp^r S^n \times S^n$ will yield the minimal $k$ which we desire. Setting $n=m=d/2$ in Theorem \ref{examples}, we see that $k_{min}=n+2=(d/2)+2.$
	
	If $d$ is odd, say $d=2c+1,$ then the optimal connected sum to consider is $\sharp^r S^{c+1} \times S^c,$ and by Theorem \ref{examples} this gives $k_{min}=c+2=\lfloor d/2 \rfloor +2.$
\end{proof}

	\begin{remark}\label{simple_optimal}
		Lemma \ref{curv_op} shows that the metric constructed in the proof of Theorem \ref{surgery} is \emph{simple}, i.e.\ the curvature operator has an orthonormal basis of decomposable eigenvectors. By \cite[Corollary 5.28]{Ch} compact Riemannian manifolds with $Ric_k>0$ and simple curvature operator satisfy $b_k(M)=0$. Hence, for metrics with simple curvature operator, the bound on $k$ in Theorem \ref{examples} is optimal if $n\neq m$ and could potentially only be lowered by $1$ if $n=m$.
	\end{remark}

%%%%%%%%%%

\bigskip\bigskip

\noindent{\it Philipp Reiser, Department of Mathematics, University of Fribourg, Switzerland,
	Email: philipp.reiser@unifr.ch }\\

\bigskip

\noindent {\it David Wraith, 
	Department of Mathematics and Statistics, 
	National University of Ireland Maynooth, Maynooth, 
	County Kildare, 
	Ireland. 
	Email: david.wraith@mu.ie.}

\end{document}